\documentclass[preprint,12pt]{elsarticle1}

\usepackage{subfig}
\usepackage{graphicx}
\usepackage{caption,color}   
\usepackage{amssymb}
\usepackage{amsthm}
\usepackage{amsmath}
\usepackage{epic}
\usepackage{setspace}
\usepackage{float}
\usepackage{multirow}
\usepackage{hyperref}
\newtheorem{thm}{Theorem}[section]
\newtheorem{cor}[thm]{Corollary}

\newtheorem{lem}[thm]{Lemma}

\newtheorem{remark}[thm]{Remark}

\numberwithin{equation}{section}
\journal{}

\begin{document}
\begin{spacing}{1.15}
\begin{frontmatter}
\title{\textbf{Localization of the clique spectral version of Zykov's theorem}}

\author[1]{Changjiang Bu\corref{mycorrespondingauthor}}
\cortext[mycorrespondingauthor]{Corresponding author}
\ead{buchangjiang@hrbeu.edu.cn}
\author[1]{Jueru Liu}
\author[1]{Haotian Zeng}
\address[1]{School of Mathematical Sciences, Harbin Engineering University, Harbin 150001, PR China}


\begin{abstract}
Zykov's theorem shows that $r$-partite Tur\'{a}n graph uniquely has the maximum number of $K_t$ among all $n$-vertex $K_{r+1}$-free graphs for $2\le t\le r$.
The clique tensor is a high-order extension of the adjacency matrix of a graph.
Yu and Peng \cite{peng1} gave a spectral version of the Zykov's theorem via clique tensor.
In this paper, we give some upper bounds on the spectral radius of the clique tensor of a graph, which can be viewed as the localizations of the spectral version of Zykov's theorem.
\end{abstract}

\end{frontmatter}

\section{Introduction}
The graphs considered throughout this paper are all simple and undirected.
For a graph $G$, if an induced subgraph of a subset of $V(G)$ is a complete graph, then the subset is called a \textit{clique}.
The \textit{clique number} of $G$ is the number of vertices of a largest clique in $G$, denoted by $\omega(G)$.
A clique is called a \textit{$t$-clique} if it has $t$ vertices.
Let $C_{t}(G)$ be the set of all $t$-cliques in $G$.
Let $\rho(G)$ denote the spectral radius of $G$.
In 2002, Nikiforov \cite{nikiforov2002} gave an upper bound on the spectral radius of graphs.

\begin{thm}\cite{nikiforov2002}\label{niki}
Let $G$ be an $n$-vertex graph with clique number $\omega$. Then
$$\rho(G)\le\sqrt{2|E(G)|\left(1-\frac{1}{w}\right)}.$$
Equality holds if and only if $G$ is a complete bipartite graph for $\omega=2$, or a complete regular $\omega$-partite graph for $\omega\ge3$ and $\omega$ divides $n$ (possibly with some isolated vertices).
\end{thm}

In fact, the above conclusion implies the concise Tur\'{a}n's theorem.

\begin{thm}\cite{turan1961research}
Let $G$ be an $n$-vertex $K_{r+1}$-free graph (i.e., containing no copy of the complete graph $K_{r+1}$). Then $$|E(G)|\le\left(1-\frac{1}{r}\right)\frac{n^2}{2}.$$ Equality holds if and only if $r$ divides $n$ and $G$ is a complete regular $r$-partite graph.
\end{thm}

The Tur\'{a}n number for a graph $F$ is the maximum number of edges in an $n$-vertex $F$-free graph. Some results on Tur\'{a}n problem and spectral Tur\'{a}n problem can be referred to \cite{alonjctb2024,cioaba,jctb2023,wilf}.
Brada$\check{\mathrm{c}}$ \cite{local1} and Malec and Tompkins \cite{1local2023} gave a localized version of concise Tur\'{a}n's theorem.
For an edge $e\in E(G)$, let $\alpha(e)$ be the order of the largest clique in $G$ containing $e$.

\begin{thm}\cite{local1,1local2023}\label{local1}
Let $G$ be an $n$-vertex graph. Then
$$\sum_{e\in E(G)}\frac{\alpha(e)}{\alpha(e)-1}\le\frac{n^2}{2}.$$
Equality holds if and only if $G$ is a complete multi-partite graph with vertex classes of equal size.
\end{thm}

Liu and Ning \cite{liulelejctb} gave an upper bound on the spectral radius of a graph in terms of the order of the largest clique containing each edge, which is a localized version of spectral Tur\'{a}n's theorem. And some other results on localization of spectral Tur\'{a}n's theorem were given in \cite{liulelearxiv}.

\begin{thm}\cite{liulelejctb}\label{liulelejctb}
Let $G$ be a graph with clique number $\omega$. Then
$$\rho(G)\le\sqrt{2\sum_{e\in E(G)}\frac{\alpha(e)-1}{\alpha(e)}}.$$
Equality holds if and only if $G$ is a complete bipartite graph for $\omega=2$, or a complete regular $\omega$-partite graph for $\omega\ge3$ and $\omega$ divides $n$ (possibly with some isolated vertices).
\end{thm}

As an edge can be viewed as an induced subgraph of a $2$-clique, the generalized Tur\'{a}n number $\mathrm{ex}(n,H,F)$ studies the maximum number of copies of subgraphs $H$ in an $n$-vertex $F$-free graph.
The famous generalized Tur\'{a}n result standing on its own is the complete determination of $\mathrm{ex}(n, K_t, K_{r+1})$ by Zykov \cite{zykov} and Erd\H{o}s \cite{erdos1}.
Subsequently, Alon and Shikhelman \cite{alon} studied the function $\mathrm{ex}(n,H,F)$, some results on the generalized Tur\'{a}n numbers can be referred to \cite{cheneujc2025,Gerbnerjctb2020,luojctb2018,maj}.
In 2024, Kirsch and Nir \cite{2local2024} proposed a localized approach to generalized Tur\'{a}n problems and gave a localized version of Zykov's theorem by assigning weights to cliques of any size.
The spectral Tur\'{a}n problems have attracted considerable attention, but there are few studies on spectral versions of generalized Tur\'{a}n numbers.

In 2023, Liu and Bu \cite{liujoco} proposed the clique tensor of a graph and gave a generalization of the spectral Mantel's theorem.
Recently, some results on the spectral version of the generalized Tur\'{a}n number were given via clique tensor \cite{liu2026dm,peng3,peng1}.
In 2025, Yu and Peng \cite{peng1} gave a spectral version of Zykov's theorem, which shows that the complete regular $r$-partite graph attains the maximum $t$-clique spectral radius among all $n$-vertex $K_{r+1}$-free graphs for $2\le t\le r$.
In 2026, a tensor's spectral bound on the clique number was given \cite{liu2026dm}, which extends Nikiforov's theorem (Theorem \ref{niki}) to clique tensors.

\begin{thm}\cite{liu2026dm}\label{liu2026dm}
Let $G$ be a graph with clique number $\omega$. For $2\le t\le \omega$,
$$\rho_{t}(G)\le \frac{t}{\omega}\binom{\omega}{t}^{\frac{1}{t}}|C_{t}(G)|^{\frac{t-1}{t}}.$$
Moreover, if $G$ is a complete regular $\omega$-partite graph for $\omega\ge t\ge 2$, then the equality is achieved in the above inequality.
\end{thm}

In this paper, we give some upper bounds on the $t$-clique spectral radius of graphs in terms of the order of the largest clique containing each clique or vertex, which can be viewed as the localizations of the spectral version of Zykov's theorem.


\section{Preliminaries}
In this section, some related definitions and lemmas are introduced.
For a positive integer $n$, let $[n]=\{1,2,\ldots,n\}$.
A $k$-order $n$-dimensional complex tensor $\mathcal{A}=(a_{i_1 i_2 \cdots i_k})$ is a multi-dimensional array with $n^k$ entries on complex number field $\mathbb{C}$, where $i_1i_2\cdots i_k\in[n]^k$.
Denote the set of $n$-dimensional complex vectors and the set of $k$-order $n$-dimensional complex tensors by $\mathbb{C}^{n}$ and $\mathbb{C}^{[k,n]}$, respectively.
For $\mathcal{A}=(a_{i_1 i_2 \cdots i_k})\in\mathbb{C}^{[k,n]}$ and $x=(x_1,\ldots,x_n)^{\mathsf{T}}\in\mathbb{C}^{n}$, $\mathcal{A}x^{k-1}$ is a vector in $\mathbb{C}^{n}$ whose $i$-th component is $$\left( \mathcal{A}x^{k-1} \right)_{i}=\sum\limits_{i_2,\ldots,i_k=1}^{n}a_{i i_2 \cdots i_k}x_{i_2}\cdots x_{i_k}.$$
A number $\lambda\in\mathbb{C}$ is called an \textit{eigenvalue} of $\mathcal{A}$ if there exists a nonzero vector $x\in\mathbb{C}^{n}$ such that $$\mathcal{A}x^{k-1}=\lambda x^{[k-1]},$$ where $x^{[k-1]}=(x_1^{k-1},\ldots,x_n^{k-1})^{\mathsf{T}}$ and $x$ is called an \textit{eigenvector} of $\mathcal{A}$ associated with $\lambda$ \cite{lim2005,qi2005}. The \textit{spectral radius} of $\mathcal{A}$ is the maximum modulus of all eigenvalues of $\mathcal{A}$, denoted by $\rho(\mathcal{A})$.

A tensor $\mathcal{A}$ is termed \textit{symmetric} if its entries remain invariant under any permutation of their indices.
Furthermore, if all entries of a tensor $\mathcal{A}$ are nonnegative, then $\mathcal{A}$ is referred to as a \textit{nonnegative tensor}.
Let $\mathbb{R}^{n}_{+}$ (resp. $\mathbb{R}^{n}_{++}$) be the set of all $n$-dimensional vectors with nonnegative (resp. positive) components.

\begin{lem}\cite{qi2013}\label{qi2013}
Let $\mathcal{A}=(a_{i_1 i_2 \cdots i_k})$ be a $k$-order $n$-dimensional symmetric nonnegative tensor. The spectral radius of $\mathcal{A}$ is equal to
$$\max\{\sum_{i_1,i_2,\ldots,i_k=1}^{n}a_{i_1 i_2 \cdots i_k}x_{i_1}x_{i_2}\cdots x_{i_k}:~ \sum_{i=1}^{n}x_i^k=1,~ (x_1,x_2,\ldots,x_n)^{\mathsf{T}}\in\mathbb{R}^{n}_{+}\}.$$
\end{lem}

For an $n$-vertex graph $G$ and an integer $t~ (2\le t\le\omega(G))$, the \textit{$t$-clique tensor} $\mathcal{A}(G)=(a_{i_1i_2\cdots i_t})$ is a $t$-order $n$-dimensional tensor, with entries \cite{liujoco}
$$a_{i_1i_2\cdots i_t}=\left\{\begin{array}{cl}
			\frac{1}{(t-1)!},&  \textnormal{if}\ \{i_1,i_2,\ldots,i_t\}\in C_{t}(G),\\
			0,& \textnormal{otherwise}.
		\end{array}\right.$$
Specifically, the $2$-clique tensor is the adjacency matrix of $G$. The spectral radius of $\mathcal{A}(G)$ is called the \textit{$t$-clique spectral radius} of $G$, denoted by $\rho_t(G)$.
It is proved that $|C_{t}(G)|\le \frac{n}{t}\rho_{t}(G)$ and equality holds if the number of $t$-cliques containing each vertex in $V(G)$ is equal \cite{liujoco}.
For $v\in V(G)$, let $c_t(v)$ be the number of $t$-cliques that contain the vertex $i$ in $G$.
Next, we describe the necessary and sufficient conditions for the equality to hold.

\begin{lem}\label{ct&rhot}
Let $G$ be an $n$-vertex graph with clique number $\omega$ and let $2\le t\le\omega$. Then
$$|C_{t}(G)|\le \frac{n}{t}\rho_{t}(G).$$
Equality holds if and only if the number of $t$-cliques containing each vertex in $V(G)$ is equal.
\end{lem}

\begin{proof}
Without loss of generality, let $V(G)=[n]$.
Let $\mathcal{A}(G)=(a_{i_1 i_2 \cdots i_t})$ be the $t$-clique tensor of $G$.
For $i\in[n]$, let $c_t(i)$ denote the number of $t$-cliques contain the vertex $i$ in $G$.
Then $$\sum_{i_2,\ldots,i_t=1}^{n}a_{i i_2 \cdots i_t}=c_t(i),~ i\in[n].$$
Let $x\in\mathbb{R}^{n}_{+}$ be a vector with entries $x_i=n^{-\frac{1}{t}}~ (i\in[n])$. From Lemma \ref{qi2013}, we know that
\begin{equation}\label{rhot}
\rho_{t}(G)\ge\sum_{i_1,i_2,\ldots,i_t=1}^{n}a_{i_1 i_2 \cdots i_t}x_{i_1}x_{i_2}\cdots x_{i_t}=\frac{t\cdot|C_t(G)|}{n}.
\end{equation}

If $|C_{t}(G)|= \frac{n}{t}\rho_{t}(G)$, then the equality holds in Eq.(\ref{rhot}). It follows that the all-one vector $\mathbf{1}\in\mathbb{R}^{n}$ is the eigenvector of $\mathcal{A}(G)$ associated with $\rho_t(G)$, i.e., $\mathcal{A}(G)\mathbf{1}^{t-1}=\rho_{t}(G)\mathbf{1}^{[t-1]}$. So, we have $$\rho_{t}(G)=\left( \mathcal{A}(G)\mathbf{1}^{t-1} \right)_i=\sum_{i_2,\ldots,i_t=1}^{n}a_{i i_2 \cdots i_t}=c_t(i),~ i\in[n],$$ which implies that the number of $t$-cliques containing each vertex in $V(G)$ is equal.

If the number of $t$-cliques containing each vertex in $V(G)$ is equal, then $\rho_{t}(G)=\frac{t\cdot|C_t(G)|}{n}$ \cite{liujoco}, completing the proof.
\end{proof}

And the spectral radius of the $t$-clique tensor of a complete $t$-partite graph was also obtained in \cite{liujoco}.

\begin{lem}\cite{liujoco}\label{111}
   Let $G$ be a complete $t$-partite graph with partition $V_1,V_2,\ldots,V_{t}$. Then
   $$\rho_{t}(G)=\left( \prod_{i=1}^{t}|V_i| \right)^{\frac{t-1}{t}}.$$ 
\end{lem}

For a vector $x=(x_1,x_2,\ldots,x_n)^{\mathsf{T}}\in\mathbb{R}^{n}$ and a set $I\subseteq[n]$, denote the product $x_I=\prod_{i\in I}x_i$.
Given a graph $G$, for  a $t$-clique $I\in C_t(G)$, let $\alpha(I)$ be the order of the largest clique in $G$ containing $I$, where $2\le t\le \omega(G)$.
For two integers $s$ and $q$ with $1\le s\le q$, define the following homogeneous polynomials $$h_{s,G}(x)=\sum_{J\in C_s(G)}x_J$$ and $$f_{s,q,G}(x)=\sum_{I\in C_q(G)}\binom{\alpha(I)}{s}^{\frac{q}{s}}\binom{\alpha(I)}{q}^{-1}x_I.$$

\begin{lem}\cite{macl}\label{macl}
For every $x\in\mathbb{R}^{n}_{+}$, then $$f_{s,q,G}(x)\le h_{s,G}(x)^{\frac{q}{s}}.$$
Moreover, equality holds for $x\in\mathbb{R}^{n}_{++}$ only when the subgraph of $G$ induced on the set of vertices that belong to an $s$-clique is a complete $l$-partite graph with parts $V_1,\ldots,V_l$, for some $l\ge q$, and $\sum_{v\in V_i}x_v=\sum_{u\in V_j}x_u$ for all $1\le i,j\le l$.
\end{lem}

For a vector $x\in\mathbb{R}^{n}$, the \textit{support} of $x$, denoted by $\mathrm{supp}(x)$, is the set of all indices corresponding to nonzero entries in $x$.
When $s=1$ and $x=(x_1,\ldots,x_n)^{\mathsf{T}}\in\mathbb{R}^{n}_{+}$ is a vector with $||x||_1=x_1+\cdots+x_n=1$, from Lemma \ref{macl}, we can get the following conclusion directly.

\begin{lem}\label{w1}
Let $G$ be an $n$-vertex graph with clique number $\omega$ and let $2\le t\le\omega$. For any vector $x\in\mathbb{R}^{n}_{+}$ with $||x||_{1}=1$,
\begin{equation*}\label{whms}
\sum_{I\in C_t(G)}\left(\alpha(I)\right)^t\binom{\alpha(I)}{t}^{-1}x_I\le1.
\end{equation*}
Equality holds if and only if the induced subgraph of $G$ on $\mathrm{supp}(x)$ is a complete $\omega$-partite graph with partition $V_1, V_2, \ldots, V_{\omega}$ satisfying $\sum_{v\in V_i}x_v=\frac{1}{\omega}$ for all $i\in[\omega]$.
\end{lem}

For a graph $G$ and a vertex $v\in V(G)$, let $\alpha(v)$ denote the order of the largest clique containing $v$ in $G$.
For a $t$-clique $I=\{i_1,i_2,\ldots,i_t\}$ in $G$, it is clear that $\alpha(I)\le\min\{\alpha(i_1),\alpha(i_2),\ldots,\alpha(i_t)\}$. Hence, we have the following conclusion.

\begin{lem}\label{w2}
Let $G$ be an $n$-vertex graph with clique number $\omega$ and let $2\le t\le\omega$. For any vector $x\in\mathbb{R}^{n}_{+}$ with $||x||_1=1$,
$$\sum_{I=\{i_1,i_2,\ldots,i_t\}\in C_t(G)} \frac{1}{t} \left(\sum_{j=1}^{t}\left(\alpha(i_j)\right)^t\binom{\alpha(i_j)}{t}^{-1}\right)x_I\le 1.$$
Equality holds if and only if the induced subgraph of $G$ on $\mathrm{supp}(x)$ is a complete $\omega$-partite graph with partition $V_1, V_2, \ldots, V_{\omega}$ satisfy $\sum_{v\in V_i}x_v=\frac{1}{\omega}$ for all $i\in[\omega]$.
\end{lem}

\section{Main results}
In this section, we obtain some upper bounds on the $t$-clique spectral radius of a graph, which are expressed by the order of the largest clique containing each clique or vertex and can be viewed as the localized versions of the spectral Zykov's theorem.

\begin{thm}\label{3}
Let $G$ be a graph with $t$-clique spectral radius $\rho_{t}(G)$ and clique number $\omega$ and let $2\le t\le \omega$. Then
$$\left( \frac{\rho_{t}(G)}{t} \right)^{t}\le \left( \sum_{I\in C_t(G)}\sqrt[t-1]{\binom{\alpha(I)}{t}\left(\alpha(I)\right)^{-t}} \right)^{t-1}.$$
Equality holds if and only if the graph obtained from $G$ by deleting edges not contained in $t$-cliques is a complete $w$-partite graph for $\omega=t$, or a complete regular $\omega$-partite graph for $\omega\ge t+1$ (possibly with some isolated vertices).
\end{thm}

\begin{proof}
For the $t$-clique tensor $\mathcal{A}(G)$ of the graph $G$, let $x=(x_1,\ldots,x_n)\in\mathbb{R}^{n}_{+}$ be an nonnegative eigenvector corresponding to $\rho_t(G)$ with $x_1^t+\cdots+x_n^t=1$. Then $$\rho_{t}(G)=\mathcal{A}(G)x^t=t\sum_{\{i_1,i_2,\ldots,i_t\}\in C_t(G)}x_{i_1} x_{i_2} \cdots x_{i_t}.$$

H\H{o}lder's inequality shows that for two nonnegative vectors $x=(x_1,\ldots,x_n)^{\mathsf{T}}$ and $y=(y_1,\ldots,y_n)^{\mathsf{T}}$, if two positive number $p$ and $q$ satisfy $\frac{1}{p}+\frac{1}{q}=1$, then $\sum_{i=1}^{n}x_i y_i\le \left(\sum_{i=1}^{n}x_i^p\right)^{\frac{1}{p}}\left(\sum_{i=1}^{n}y_i^q\right)^{\frac{1}{q}}$,
the equality holds if and only if $x$ and $y$ are proportional.
Thus, we have
\begin{align*}
\rho_{t}(G)&=t \sum_{\{i_1,i_2,\ldots,i_t\}=I\in C_t(G)} \left(\frac{\binom{\alpha(I)}{t}}{\left(\alpha(I)\right)^t}\right)^{\frac{1}{t}} \left(\frac{\left(\alpha(I)\right)^t}{\binom{\alpha(I)}{t}}\right)^{\frac{1}{t}}x_{i_1} x_{i_2} \cdots x_{i_t}\\
&\le t \left(\sum_{I\in C_t(G)}\left(\frac{\binom{\alpha(I)}{t}}{\left(\alpha(I)\right)^t}\right)^{\frac{1}{t-1}}\right)^{\frac{t-1}{t}} \left(\sum_{\{i_1,i_2,\ldots,i_t\}=I\in C_t(G)}\frac{\left(\alpha(I)\right)^t}{\binom{\alpha(I)}{t}}x_{i_1}^{t} x_{i_2}^{t}\cdots x_{i_t}^{t}\right)^{\frac{1}{t}}.
\end{align*}
Since $x_1^t+x_2^t+\cdots+x_n^t=1$, by Lemma \ref{w1}, we have
\begin{equation}\label{rho}
\rho_{t}(G)\le t \left(\sum_{I\in C_t(G)}\left(\frac{\binom{\alpha(I)}{t}}{\left(\alpha(I)\right)^t}\right)^{\frac{1}{t-1}}\right)^{\frac{t-1}{t}}.
\end{equation}

When $\omega = t$, if the graph $G'$ obtained from $G$ by deleting edges not contained in $t$-cliques is a complete $t$-partite graph, we have $\alpha(I)=t$ for every $I\in C_{t}(G)$. Then the Inequality (\ref{rho}) can be simplified to $$\rho_{t}(G)\le t \left(\sum_{I\in C_t(G)}\left(\frac{\binom{\alpha(I)}{t}}{\left(\alpha(I)\right)^t}\right)^{\frac{1}{t-1}}\right)^{\frac{t-1}{t}}=|C_t(G)|^{\frac{t-1}{t}}.$$
By Lemma \ref{111}, $\rho_t(G)=\rho_t(G')=\left( \prod_{i=1}^{t}|V_i| \right)^{\frac{t-1}{t}}=|C_t(G)|^{\frac{t-1}{t}}$, the equality in Eq. (\ref{rho}) holds.

When $\omega \geq t+1$, if the graph $G'$ obtained from $G$ by deleting edges not contained in $t$-cliques is a complete regular $\omega$-partite graph, we have $\alpha(I)=w$ for every $I\in C_{t}(G)$. Then the Inequality (\ref{rho}) can be simplified to $$\rho_{t}(G)\le t \left(\sum_{I\in C_t(G)}\left(\frac{\binom{\alpha(I)}{t}}{\left(\alpha(I)\right)^t}\right)^{\frac{1}{t-1}}\right)^{\frac{t-1}{t}}
=\frac{t}{\omega}\binom{\omega}{t}^{\frac{1}{t}}|C_{t}(G)|^{\frac{t-1}{t}}.$$ By Theorem \ref{liu2026dm}, $\rho_t(G)=\rho_t(G')=\frac{t}{\omega}\binom{\omega}{t}^{\frac{1}{t}}|C_{t}(G)|^{\frac{t-1}{t}}$, the equality in Eq. (\ref{rho}) holds.

Next, we characterize all graphs attaining equality in Inequality (\ref{rho}).
According to the proof above, by H\H{o}lder's inequality, equality in Eq. (\ref{rho}) holds if and only if $$\sum_{\{i_1,i_2,\ldots,i_t\}=I\in C_t(G)}\frac{\left(\alpha(I)\right)^t}{\binom{\alpha(I)}{t}}x_{i_1}^{t} x_{i_2}^{t}\cdots x_{i_t}^{t}=1$$ and for each $t$-clique $I=\{i_1,i_2,\ldots,i_t\}\in C_{t}(G)$, $$x_{i_1}x_{i_2}\cdots x_{i_t}=c\left(\frac{\binom{\alpha(I)}{t}}{\left(\alpha(I)\right)^t}\right)^{\frac{2}{t}}$$ for some constant $c>0$.
By Lemma \ref{w1} and $x_i>0$ for any vertex $i$ contained in a $t$-clique of $G$, the equality is equivalent to the following:
\begin{enumerate}
    \item[(1)] The graph $G'$ obtained from $G$ by deleting edges not contained in $t$-cliques is a complete $\omega$-partite graph (possibly with some isolated vertices), and its vertex classes $V_1,V_2,\ldots,V_{\omega}$ satisfy $\sum_{v\in V_i}x_v^{t}=\frac{1}{\omega}$ for all $i\in[\omega]$.
    \item[(2)] And for each $t$-clique $\{i_1,i_2,\cdots,i_t\}\in C_t(G)$, $x_{i_1} x_{i_2}\cdots x_{i_t} = c'$ for some constant $c'>0$.
\end{enumerate}


From the above items, when $\omega = t$, the graph $G'$ obtained from $G$ by deleting edges not contained in $t$-cliques is a complete $\omega$-partite graph. When $\omega \geq t+1$, let $G'$ be the complete $\omega$-partite graph. For any $i \ne j$ and $u \in V_i$, $v \in V_j$, by item (2), we have $x_{u} x_{i_2}\cdots x_{i_t}= c' =x_{v} x_{i_2}\cdots x_{i_t}$ for any $i_2,\cdots,i_t\notin\left(V_i \cup V_j\right)$ and $i_2,\cdots,i_t$ respectively come from other $t-1$ different partitions. Then $x_u = x_v$, and therefore $|V_1| = |V_2| = \cdots = |V_\omega|$, i.e. $G'$ is a complete regular $\omega$-partite graph.
\end{proof}

\begin{remark}
When $t=2$, the conclusion in Theorem \ref{3} shows that $$\rho^2(G)\le2\sum_{e\in E(G)}\frac{\alpha(e)-1}{\alpha(e)},$$ which is the localized version of spectral Tur\'{a}n's theorem given by Liu and Ning \cite{liulelejctb} (i.e., Theorem \ref{liulelejctb}).
And $\alpha(I)\le \omega(G)$ for any $I\in C_t(G)$, then from Theorem \ref{3}, we have $$\left( \frac{\rho_t(G)}{t} \right)^{t}\le \frac{1}{\omega(G)^t}\binom{\omega(G)}{t}|C_t(G)|^{t-1},$$ which implies the inequality in Theorem \ref{liu2026dm}.
\end{remark}


For any $t$-clique $I=\{i_1,i_2,\ldots,i_t\}\in C_t(G)$, since $\alpha(I)\le \alpha(i_j)$ for $j=1,\ldots,t$, we can get the following conclusion.

\begin{cor}\label{vertex}
Let $G$ be a graph with $t$-clique spectral radius $\rho_{t}(G)$ and clique number $\omega$ and let $2\le t\le\omega$. Then
$$\left( \rho_{t}(G)\right)^{t}\le t \left( \sum_{v\in V(G)} c_{t}(v) \sqrt[t-1]{\binom{\alpha(v)}{t}\left(\alpha(v)\right)^{-t}} \right)^{t-1}.$$
Equality holds if and only if the graph obtained from $G$ by deleting edges not contained in $t$-cliques is a complete $\omega$-partite graph for $\omega=t$, or a complete regular $\omega$-partite graph for $\omega\ge t+1$ (possibly with some isolated vertices).
\end{cor}

\begin{proof}
Observe that for any $t$-clique $I=\{i_1,i_2,\ldots,i_t\}\in C_t(G)$, $\alpha(I)\le \alpha(i_j)$ for $j=1,\ldots,t$. Thus, by Theorem \ref{3}, we have
\begin{align*}
\left( \frac{\rho_t(G)}{t} \right)^{t}& \le \left( \sum_{I\in C_t(G)} \sqrt[t-1]{\frac{\binom{\alpha(I)}{t}}{\left(\alpha(I)\right)^t}} \right)^{t-1}\\
& \le \frac{1}{t^{t-1}}\left( \sum_{\{i_1,i_2,\ldots,i_t\}=I\in C_t(G)}\sum_{j=1}^{t} \sqrt[t-1]{ \frac{\binom{\alpha(i_j)}{t}}{\left(\alpha(i_j)\right)^t}} \right)^{t-1}\\
& = \frac{1}{t^{t-1}} \left( \sum_{v\in V(G)} c_{t}(v) \sqrt[t-1]{\frac{\binom{\alpha(v)}{t}}{\left(\alpha(v)\right)^t}} \right)^{t-1}.
\end{align*}

Next,we consider the equality. By Theorem \ref{3}, the equality holds in the last inequality if and only if the graph obtained from $G$ by deleting edges not contained in $t$-cliques is a complete $t$- partite graph for $\omega=t$, or a complete regular $\omega$-partite graph for $\omega\ge t+1$ (possibly with some isolated vertices). In this condition, for any $t$-clique $I=\{i_1,i_2,\ldots,i_t\}\in C_t(G)$, $\alpha(I) =\alpha(i_j)=\omega$ for $j=1,\ldots,t$,  the equalities also hold in the further inequality, completing the proof.
\end{proof}

When $t=2$, Corollary \ref{vertex} is a localized version of Wilf's equality given in \cite{liulelearxiv}.
The following result can provide a new upper bound on $t$-clique spectral radius based on $\alpha(v)$.

\begin{thm}\label{le}
Let $G$ be an $n$-vertex graph with clique number $\omega$ and let $2\le t\le\omega$. Then
$$\sum_{v\in V(G)}\frac{c_t(v)}{t}\sqrt[t-1]{\binom{\alpha(v)}{t}\left(\alpha(v)\right)^{-t}} \le \left( \sum_{v\in V(G)}\sqrt[t-1]{\binom{\alpha(v)}{t}\left(\alpha(v)\right)^{-t}} \right)^{t}.$$
Equality holds if and only if the graph obtained from $G$ by deleting edges not contained in $t$-cliques is a complete regular $\omega$-partite graph (possibly with some isolated vertices).
\end{thm}

\begin{proof}
A specific form of Muirhead's inequality states that for positive real number $z_1,z_1,\ldots,z_m$,
$$\sum_{k=1}^{m}\frac{1}{z_k}\le\left(\sum_{k=1}^{m}z_k^{m-1}\right)\frac{1}{\prod_{j=1}^{m}z_j},$$ and equality holds if and only if $z_1=\cdots=z_m$.
Then, we have
\begin{align*}
	 \sum_{v\in V(G)}\frac{c_t(v)}{t}\sqrt[t-1]{\frac{\binom{\alpha(v)}{t}}{\left(\alpha(v)\right)^t}}
= & \frac{1}{t} \sum_{\{i_1,i_2,\ldots,i_t\}\in C_t(G)}\left( \sum_{j=1}^{t} \sqrt[t-1]{\frac{\binom{\alpha(i_j)}{t}}{\left(\alpha(i_j)\right)^t}} \right)\\
\le & \frac{1}{t} \sum_{\{i_1,i_2,\ldots,i_t\}\in C_t(G)} \left( \left( \sum_{j=1}^{t} \frac{\left(\alpha(i_j)\right)^t}{\binom{\alpha(i_j)}{t}} \right) \prod_{j=1}^{t}\sqrt[t-1]{\frac{\binom{\alpha(i_j)}{t}}{\left(\alpha(i_j)\right)^t}} \right)\\
= & t \sum_{\{i_1,i_2,\ldots,i_t\}\in C_t(G)} \left( \left( \frac{1}{t^2} \sum_{j=1}^{t} \frac{\left(\alpha(i_j)\right)^t}{\binom{\alpha(i_j)}{t}} \right) \prod_{j=1}^{t}\sqrt[t-1]{\frac{\binom{\alpha(i_j)}{t}}{\left(\alpha(i_j)\right)^t}} \right).
\end{align*}
Let $V(G)=\{v_1,v_2,\ldots,v_n\}$. We construct a vector
$$x=\left( \sqrt[t-1]{\frac{\binom{\alpha(v_1)}{t}}{\left(\alpha(v_1)\right)^t}}, \sqrt[t-1]{\frac{\binom{\alpha(v_2)}{t}}{\left(\alpha(v_2)\right)^t}}, \ldots, \sqrt[t-1]{\frac{\binom{\alpha(v_n)}{t}}{\left(\alpha(v_n)\right)^t}} \right)^{\mathsf{T}}\in\mathbb{R}^{n}_{+}.$$
And let $\mathcal{W}(G)=(w_{i_1i_2\cdots i_t})$ be the weight $t$-clique tensor of $G$ with $w_{i_1i_2\cdots t_t}=\omega_{i_1i_2\cdots i_t}\frac{1}{(t-1)!}$ if $\{i_1,i_2,\ldots,i_t\}$ is a $t$-clique in $G$ and $w_{i_1i_2\cdots t_t}=0$ otherwise, where $\omega_{i_1i_2\cdots i_t}=\frac{1}{t^2} \sum_{j=1}^{t} \frac{\left(\alpha(i_j)\right)^t}{\binom{\alpha(i_j)}{t}}$.
Then
\begin{align*}
	\sum_{v\in V(G)}\frac{c_t(v)}{t}\sqrt[t-1]{\frac{\binom{\alpha(v)}{t}}{\left(\alpha(v)\right)^t}}\le \mathcal{W}(G)x^{t}.
\end{align*}
Let $y=\frac{x}{||x||_{1}}\in\mathbb{R}^{n}_{+}$. Then $||y||_{1}=\sum_{i=1}^{n}y_i=1$. By Lemma \ref{w2}, we have
$$\mathcal{W}(G)x^t=\mathcal{W}(G)y^t \cdot ||x||_{1}^{t}\le ||x||_{1}^{t}.$$
Thus,
$$	\sum_{v\in V(G)}\frac{c_t(v)}{t}\sqrt[t-1]{\frac{\binom{\alpha(v)}{t}}{\left(\alpha(v)\right)^t}} \le \mathcal{W}(G)x^{t} \le ||x||_{1}^{t} = \left( \sum_{v\in V(G)}\sqrt[t-1]{\frac{\binom{\alpha(v)}{t}}{\left(\alpha(v)\right)^t}} \right)^{t}.$$

Next, we consider the equality in the above inequality.
If the graph obtained from $G$ by deleting edges not contained in $t$-cliques is a complete regular $\omega$-partite graph, we will verify that equality holds.
Without loss of generality, assume that $\omega|n$ and $G$ is a complete regular $\omega$-partite graph.
Then $\alpha(v)=\omega$ and $c_t(v)=\binom{\omega-1}{t-1}\left( \frac{n}{\omega} \right)^{t-1}$ for any $v\in V(G)$.
Therefore, we can proof that the equality holds.

Conversely, if equality holds, then by Lemma \ref{w2}, $G[\mathrm{supp}(y)]$ is a complete $\omega$-partite graph with partition $V_1, V_2, \ldots, V_{\omega}$ satisfy $\sum_{v\in V_i}y_v=\frac{1}{\omega}$ for all $i\in[\omega]$.
By Muirhead's inequality, we know that $\alpha(v)=\omega$ for each $v$ contained in $t$-cliques.
Hence, for any $i\in[\omega]$,
$$\frac{1}{\omega}=\frac{1}{||x||_1}\sum_{v\in V_i}\sqrt[t-1]{\frac{\binom{\alpha(v)}{t}}{\left(\alpha(v)\right)^t}}=\frac{|V_i|}{||x||_1}\sqrt[t-1]{\frac{\binom{\omega}{t}}{\omega^t}},$$
i.e., $|V_1|=|V_2|=\cdots=|V_{\omega}|$.
Consequently, the graph obtained from $G$ by deleting edges not contained in $t$-cliques is a complete regular $\omega$-partite graph.
\end{proof}

By the above theorem, a new upper bound for $\rho_{t}(G)$ is given as follows. When $t=2$, this is the conclusion given in \cite{liulelearxiv}, which serves as a local version of Wilf's theorem.

\begin{cor}\label{123}
Let $G$ be a graph with $t$-clique spectral radius $\rho_{t}(G)$ and clique number $\omega$ and let $2\le t\le\omega$. Then
$$\rho_{t}(G)\le t\left( \sum_{v\in V(G)} \sqrt[t-1]{\binom{\alpha(v)}{t}\left(\alpha(v)\right)^{-t}} \right)^{t-1} .$$
Equality holds if and only if the graph obtained from $G$ by deleting edges not contained in $t$-cliques is a complete regular $\omega$-partite graph (possibly with some isolated vertices).
\end{cor}

\begin{proof}
By Corollary \ref{vertex} and Theorem \ref{le}, we have
\begin{align*}
	\left( \frac{\rho_{t}(G)}{t} \right)^{t} & \le \frac{1}{t^{t-1}} \left( \sum_{v\in V(G)} c_{t}(v) \sqrt[t-1]{\frac{\binom{\alpha(v)}{t}}{\left(\alpha(v)\right)^t}} \right)^{t-1}\\
	& = \left( \sum_{v\in V(G)} \frac{c_{t}(v)}{t} \sqrt[t-1]{\frac{\binom{\alpha(v)}{t}}{\left(\alpha(v)\right)^t}} \right)^{t-1}\\
	& \le \left( \sum_{v\in V(G)}\sqrt[t-1]{\frac{\binom{\alpha(v)}{t}}{\left(\alpha(v)\right)^t}} \right)^{t(t-1)}.
\end{align*}
By Corollary \ref{vertex} and Theorem \ref{le}, we can characterize all graphs attaining equalities in the above inequalities.
\end{proof}

Combining Lemma \ref{ct&rhot} and  Corollary \ref{vertex}, a upper bound on the number of $t$-cliques can be given directly.

\begin{cor}\label{clique vertex}
Let $G$ be an $n$-vertex graph with clique number $\omega$ and
let $2\leq t\leq \omega$. Then
$$|C_t(G)|\leq n\left( \frac{1}{t}\sum_{v\in V(G)}c_{t}(v) \sqrt[t-1]{{\binom{\alpha(v)}{t}}{\left(\alpha(v)\right)^{-t}}} \right)^{\frac{t-1}{t}}.$$
Equality holds if and only if the graph obtained from $G$ by deleting edges not contained in $t$-cliques is a complete regular $\omega$-partite graph (possibly with some isolated vertices).
\end{cor}
Combining Theorem \ref{le} and  Corollary \ref{clique vertex}, we can get a weaker upper bound on the number of $t$-cliques but without $c_{t}(v)$.

\begin{cor}\label{124}
Let $G$ be an $n$-vertex graph with clique number $\omega$ and let $2\le t\le\omega$. Then
$$|C_t(G)|\leq n\left( \sum_{v\in V(G)} \sqrt[t-1]{\binom{\alpha(v)}{t}\left(\alpha(v)\right)^{-t}} \right)^{t-1}.$$
Equality holds if and only if the graph obtained from $G$ by deleting edges not contained in $t$-cliques is a complete regular $\omega$-partite graph (possibly with some isolated vertices).
\end{cor}

\begin{remark}
From Corollary \ref{124}, by H\H{o}lder's inequality, we have
\begin{align*}
|C_t(G)|&\leq n\left( \sum_{v\in V(G)} \sqrt[t-1]{\frac{\binom{\alpha(v)}{t}}{\left(\alpha(v)\right)^t}} \right)^{t-1}\\
&\leq n\left(\left(\sum_{v\in V(G)}1^{\frac{t-1}{t-2}} \right)^{\frac{t-2}{t-1}}\left(\sum_{v\in V(G)} {\frac{\binom{\alpha(v)}{t}}{\left(\alpha(v)\right)^t}}\right)^{\frac{1}{t-1}}\right)^{t-1}\\
&=n^{t-1}\left(\sum_{v\in V(G)} {\frac{\binom{\alpha(v)}{t}}{\left(\alpha(v)\right)^t}}\right),
\end{align*}
which is a vertex-based localized Zykov's inequality given in \cite{dian4}.
By H\H{o}lder's inequality, equality holds in the second inequality if and only if $\frac{\binom{\alpha(v)}{t}}{\left(\alpha(v)\right)^t}=c$ for each vertex $v$ contained in $t$-cliques and some constant $c>0$, which is equivalent to the order of the largest clique containing each vertex in $G$ is equal.
Thus, if there exist two vertices $u,v\in V(G)$ with $\alpha(u)\ge t$, $\alpha(v)\ge t$ and $\alpha(u)\ne \alpha(v)$, then the upper bound of the number $t$-cliques given in Corollary \ref{124} is strictly less than that in \cite{dian4}.
\end{remark}

\vspace{3mm}



\section*{References}
\bibliographystyle{plain}
\bibliography{spbib}
\end{spacing}
\end{document}